\documentclass[12pt]{amsart}
\usepackage{amscd,amssymb,epsfig,mathtools}
\calclayout

\numberwithin{equation}{section} 

\newcommand{\ud}{\,d} 
 
\newcommand{\R}{\mathbb{R}}

\newcommand\cof{\operatorname{cof}}

\newcommand\tr{\operatorname{tr}}

\renewcommand{\div}{\operatorname{div}}

\newcommand{\tir}[1]{\ensuremath{\overline {#1}}} 
\newtheorem{thm}{Theorem}[section] 
 
\newtheorem{lemma}[thm]{Lemma}

\newtheorem{rem}[thm]{Remark}

\def\whsq{\vbox to 5.8pt 
{\offinterlineskip\hrule 
\hbox to 5.8pt{\vrule height 
5.1pt\hss\vrule height 5.1pt}\hrule}}

\def\<{\langle} 
\def\>{\rangle} 
\def\PP{{\mathop{{\rm I}\kern-.2em{\rm P}}\nolimits}} 
\def\FF{{\mathop{{\rm I}\kern-.2em{\rm F}}\nolimits}}   
\def\ZZ{{\mathop{{\rm I}\kern-.2em{\rm Z}}\nolimits}} 
 
\newlength{\sidemargin} 
\begin{document}
\title[]{
Quadratic Mixed finite element approximations of the Monge-Amp\`ere equation in 2D}

\thanks{The author was partially supported by NSF DMS grant No 1319640.}
\author{Gerard Awanou}
\address{Department of Mathematics, Statistics, and Computer Science, M/C 249.
University of Illinois at Chicago, 
Chicago, IL 60607-7045, USA}
\email{awanou@uic.edu}  
\urladdr{http://www.math.uic.edu/\~{}awanou}

\maketitle

\begin{abstract}
We give error estimates for a mixed finite element approximation of the two-dimensional elliptic Monge-Amp\`ere equation with the unknowns approximated by Lagrange finite elements of degree two.  The variables in the formulation are the scalar variable and the Hessian matrix. 
\end{abstract}

\section{Introduction}
Let $\Omega$ be a convex polygonal domain of $\R^2$  with boundary $\partial \Omega$. 
We are interested in a mixed finite element method for the nonlinear elliptic Monge-Amp\`ere equation: find a smooth 
convex function $u$ such that
\begin{align} \label{m1}
\begin{split}
\det (D^2 u) & = f \, \text{in} \, \Omega\\
u & = g \, \text{on} \, \partial \Omega.
\end{split}
\end{align}
For $u \in C^2(\Omega)$, $D^2 u=\bigg( (\partial^2 u) / (\partial x_i \partial x_j)\bigg)_{i,j=1,\ldots, 2} $ denotes the Hessian matrix of $u$ and  $\det D^2 u$ denotes its determinant. The function $f$ defined on $\Omega$ is assumed to satisfy $f \geq c_0 >0$ for a constant $c_0 >0$ and we assume that $g \in C(\partial \Omega)$ can be extended to a function $\tilde{g} \in C(\tir{\Omega})$ which is  convex in $\Omega$.

We consider a mixed formulation with unknowns the scalar variable $u$ and the Hessian $D^2 u$. The scalar variable and the components of the Hessian are approximated by Lagrange elements of degree $k \geq 2$.  The method considered in this paper was analyzed from different point of views in \cite{Neilan2013} and  \cite{AwanouLiMixed1} for smooth solutions of \eqref{m1}. In both \cite{Neilan2013} and  \cite{AwanouLiMixed1} the convergence of the method for Lagrange elements of degree $k=1$ and $k=2$ was left unresolved. In this paper we resolve this issue for quadratic elements. 

The ingredients of our approach consist in a fixed point argument, which yields the convergence of a time marching method, a ''rescaling argument'', i.e. the solution of a rescaled version of the equation, and 
the continuity of the eigenvalues of a matrix as a function of its entries. This is the same approach we took in the case of the standard finite element discretization of the Monge-Amp\`ere equation \cite{Awanou-Std01}. 

With the mixed methods, one can apply directly Newton's method to the discrete nonlinear problem and still have numerical evidence of convergence to a larger class of non smooth solutions than what is possible with the standard finite element discretization. We refer to \cite{Neilan2013,Lakkis11b} for the numerical results. Moreover with the standard finite element discretization \cite{Awanou-Std01}, convexity must be enforced weakly through appropriate iterative methods. Although the number of unknowns in the mixed methods is higher, in \cite{Neilan2013,Lakkis11b} the discrete Hessian was eliminated from the discrete equations in the implementation. However, as observed in \cite{AwanouLiMixed1} this prevents numerical convergence for smooth solutions when linear elements are used to approximate all the unknowns. We note that in \cite{Neilan2013} a stabilized method was proposed which works numerically for non smooth solutions in two dimension. It consists in using piecewise constants for the discrete Hessian and linear elements for the scalar variable. The analysis for smooth solutions of the lowest order methods discussed in \cite{AwanouLiMixed1,Neilan2013} cannot be done with the approach of this paper. The techniques used in this paper generalize to the three-dimensional problem but only for $k \geq 3$. 
It should be possible to extend the approach taken in this paper to the formulation where discontinuous elements are used to approximate the unknowns \cite{Neilan2013}. Numerical results reported in \cite{Neilan2013} indicate the latter approach could lead to a less accurate approximation of the Hessian. For simplicity, and to focus on the methodology we present, we do not consider such an extension in this paper.



We organize the paper as follows. In the second section we introduce some notation and preliminaries. The error analysis of the mixed method is done in section \ref{error}.   


\section{Notation and Preliminaries} \label{notation}

 We use the usual notation $L^p(\Omega), 2 \leq p \leq \infty$ for the Lebesgue spaces and $H^s(\Omega), 1 \leq s < \infty$ for the Sobolev spaces of elements of  $L^2(\Omega)$ with weak derivatives of order less than or equal to $s$ in $L^2(\Omega)$. We recall that $H_0^1(\Omega)$ is the subset of $H^1(\Omega)$
of elements with vanishing trace on $\partial \Omega$. We also recall that $W^{s,\infty}(\Omega)$ is the Sobolev space of functions with weak derivatives of order less than or equal to $s$ in $L^{\infty}(\Omega)$.
For a given normed space $X$, we denote by $X^{2}$ the space of vector fields with components in $X$ and by  $X^{2 \times 2}$ the space of matrix fields with each component in $X$.

The norm in $X$ is denoted by $|| . ||_X$ and we omit the subscript $\Omega$ and superscripts $2$ and $2 \times 2$ when it is clear from the context. 
The inner product in $L^2(\Omega), L^2(\Omega)^2$, and $L^2(\Omega)^{2 \times 2}$ is denoted by $(,)$ and we use $\< , \>$ for the inner product on  $L^2(\partial \Omega)$ and $L^2(\partial \Omega)^2$. For inner products on subsets of $\Omega$, we will simply append the subset notation. 

We denote by $n$ the unit outward normal vector to $\partial \Omega$.  We recall that for a matrix $A$, $A_{ij}$ denote its entries and the cofactor matrix of
$A$, denoted $\cof A$, is the matrix with entries $(\cof A)_{ij}=(-1)^{i+j} \det(A)_i^j$ where $\det(A)_i^j$ is the determinant of the matrix obtained from $A$ by deleting its $i$th row and its $j$th column. For two matrices $A=(A_{ij})$ and $B=(B_{ij})$,  $A: B=\sum_{i,j=1}^2 A_{ij} B_{ij}$ denotes their Frobenius inner product. A quantity which is constant is simply denoted by $C$.

For a scalar function $v$ we denote by $D v$ its gradient vector and recall that $D^2 v$ denotes the Hessian matrix of second order derivatives.  
The divergence of a matrix field is understood as the vector obtained by taking the divergence of each row.

In this section and section \ref{error} we assume that \eqref{m1} has a solution which is sufficiently smooth. Put $\sigma=D^2 u$. Then the unique convex solution $u \in H^3(\Omega)$ of \eqref{m1} satisfies  the following mixed problem:
Find $(u,\sigma) \in H^2(\Omega) \times H^1(\Omega)^{2 \times 2}$ such that
\begin{align} \label{m11}
\begin{split}
(\sigma,\tau) + (\div \tau, D u) - \< D u, \tau n \> & = 0, \forall \tau \in H^1(\Omega)^{2 \times 2} \\
(\det \sigma,v) & = (f,v), \forall v \in H_0^1(\Omega)\\
u & = g \, \text{on} \, \partial \Omega.
\end{split}
\end{align} 
It is proved in \cite{AwanouLiMixed1} that the above variational problem is well defined.

\subsection{Discrete variational problem}
We denote by  $\mathcal{T}_h$ a triangulation of $\Omega$ into simplices $K$ and assume that $\mathcal{T}_h$ is quasi-uniform. We denote by $V_h$  the standard Lagrange finite element space of degree $k \geq 2$ and denote by $\Sigma_h$ the space of symmetric matrix fields with components in the Lagrange finite element space of degree $k \geq 2$. Let $I_h$ denote the  standard Lagrange interpolation operator from $H^s(\Omega), s \geq k+1$ into the space $V_h$. We use as well the notation $I_h$ for the matrix version of the Lagrange interpolation operator mapping  $H^s(\Omega)^{2 \times 2}$, for $s \geq k+1$, into $\Sigma_h$.
We consider the problem:
find $(u_h, \sigma_h) \in V_h \times \Sigma_h$ such that
\begin{align} \label{m11h}
\begin{split}
(\sigma_h,\tau) + (\div \tau, D u_h) - \< D u_h, \tau n \> & = 0, \forall \tau \in \Sigma_h\\
(\det \sigma_h, v) & = ( f, v), \, \forall v \in V_h \cap H_0^1(\Omega)\\
u_h & = g_h \, \text{on} \, \partial \Omega,
\end{split}
\end{align}
where $g_h=I_h \tilde{g}$. 
It follows from the analysis in \cite{Neilan2013,AwanouLiMixed1} that \eqref{m11h} is well-posed for $k \geq 3$ and error estimates were given. In section \ref{error} we give an error analysis valid for $k \geq 2$.

For $v_h \in V_h$, we will make the abuse of notation of using $D^2 v_h$ to denote the Hessian of $v_h$ computed element by element. We will need the broken Sobolev norm
$$
||v||_{H^k(\mathcal{T}_h)} = \bigg( \sum_{K \in \mathcal{T}_h} ||v||^2_{H^k(K)}
\bigg)^{\frac{1}{2}}.
$$

\subsection{Properties of the Lagrange finite element spaces}
We recall some properties of the Lagrange finite element space of degree $k \geq 1$ that will be used in this paper. They can be found in  \cite{Brenner02,Bramble86}. We have

Interpolation error estimates.
\begin{align} \label{interpol}
\begin{split}
||v - I_h v||_{H^j} & \leq C h^{k+1-j} ||v||_{H^{k+1}}, \forall v \in H^{s}(\Omega),  j=0,1, \, \\
||v - I_h v||_{L^{\infty}} & \leq C h^{k} |v|_{H^{k+1}},  \forall v \in H^{s}(\Omega).
\end{split}
\end{align}
Inverse inequalities
\begin{align} 
||v||_{L^{\infty}} & \leq C h^{-1} ||v||_{L^2}, \forall v \in V_h \label{inverse0} \\
||v||_{H^1} & \leq C h^{-1} ||v||_{L^2}, \forall v \in V_h \label{inverse1} \\
||v||_{H^{k+1}(\mathcal{T}_h)} & \leq C h^{-k-1} ||v||_{L^{2}}, \forall v \in V_h. \label{inverse2}
\end{align}
Scaled trace inequality
\begin{align} 
||v||_{L^2(\partial \Omega)} &\leq C h^{-\frac{1}{2}} ||v||_{L^2},\ \forall v \in V_h. \label{trace-inverse}
\end{align}

\subsection{Algebra with matrix fields}

We collect in the following lemma some properties of matrix fields, the proof of which can be found in \cite{AwanouLiMixed1,AwanouPseudo10}. 

\begin{lemma}
 For $K \in \mathcal{T}_h$ and $u, v \in C^2(K)$ we have
\begin{align} \label{mean-v}
\det D^2 u - \det D^2 v = \cof(t D^2 u + (1-t) D^2 v): (D^2 u - D^2 v), 
\end{align}
for some $t \in [0,1]$. It can be shown that $t=1/2$, \cite{Brenner2010b}.

For two $2 \times 2$ matrix fields $\eta$ and $\tau$
\begin{align}
||\cof (\eta):\tau||_{L^2} & \leq C ||\eta||_{L^{\infty}}^{} ||\tau||_{L^2}, \label{cof-est} \\
\cof (\eta) - \cof (\tau) & = \cof(\eta-\tau). \label{cof-mv}
\end{align}
\end{lemma}

\subsection{Continuity of the eigenvalues of a matrix as a function of its entries}
Let $\lambda_1(A)$ and $\lambda_2(A)$ denote the smallest and largest eigenvalues of the symmetric matrix $A$. We have

\begin{lemma} [\cite{Awanou-Std01}, Lemma 3.1] \label{lem-1}
There exists constants $m, M >0$ independent of $h$ and a constant $C_{conv} > 0$ independent of $h$ such that for all $v_h \in V_h$ with
$v_h=g_h$ on $\partial \Omega$ and 
$$
||v_h-I_h u||_{H^1} < C_{conv} h^{2},
$$
we have
$$
m \leq \lambda_1(\cof D^2 v_h(x)) \leq \lambda_2(\cof D^2 v_h(x))  \leq M, \forall x \in K, K \in \mathcal{T}_h.
$$ 
\end{lemma}

The following lemma was used implicitly in \cite{AwanouPseudo10,Awanou-Std01,Awanou-Std-fd}. 
\begin{lemma} \label{time-trick}
Assume $0 < \alpha < 1$ and $\alpha^{} \leq (m+M)/(2m)$ for constants $m, M >0$. 
Let $B$ be a symmetric matrix field such that
$$
0 < m \alpha^{} \leq \lambda_1(B(x)) \leq \lambda_2 (B(x)) \leq M \alpha^{ }, \forall x \in \Omega.
$$
Then for $\nu=(m+M)/2$ 
$$
\gamma \equiv \sup_{v, w \in V_h \atop |v|_{H^1}=1, |w|_{H^1}=1 } \bigg|  (D v , D w) - \frac{1}{\nu} (B D v, D w)\bigg|,
$$
satisfies $0 < \gamma < 1$.
\end{lemma} 
\begin{proof} 
Since $\lambda_1(B)$ and $\lambda_2(B)$ are the minimum
and maximum respectively of the Rayleigh quotient $((B z) \cdot z)/||z||^2$, where $||z||$ denotes the Euclidean norm of $\mathbb{R}^2$, we have
for $x \in \Omega$
$$
m \alpha^{} ||z||^2 \leq (B(x) z)\cdot z \leq  M  \alpha ||z||^2, z \in \mathbb{R}^2. 
$$
This implies
\begin{equation*}
m \alpha^{} |w|_{H^1}^2 \leq \int_{\Omega} [ B(x) D w(x)] \cdot D w(x) \, \ud x \leq  M \alpha^{} |w|_{H^1}^2, w \in V_h. 
\end{equation*}
If we assume in addition that $|w|_{H^1}=1$, we get
\begin{equation*}
m \alpha^{} \leq \int_{\Omega} [ B(x) D w(x)] \cdot D w(x) \, \ud x \leq  M \alpha^{}, w \in V_h. 
\end{equation*}
It follows that
\begin{equation*}
(1-\frac{M \alpha^{}}{\nu})  \leq \int_{\Omega} [I - \frac{1}{\nu} B(x) D w(x)] \cdot D w(x) \, \ud x \leq  (1-\frac{m \alpha^{}}{\nu}) , w \in V_h. 
\end{equation*}
Since $\nu= (m+M)/2$, we have
\begin{align*}
1- \frac{\alpha^{} M}{  \nu} & = \frac{m+M - 2 M \alpha^{}}{m+M} < 1 \\
1- \frac{\alpha^{} m}{ \nu} & =  \frac{m+M - 2 m \alpha^{}}{m+M} < 1.
\end{align*}
If we define 
$$
\beta \equiv \sup_{v \in V_h,  |v|_{H^1}=1 } \bigg|(D v , D v) - \frac{1}{\nu} (B D v, D v)\bigg|,
$$
by the assumptions on $\alpha$, we have
$$
0 < \beta  < 1.
$$
We can define a bilinear form on $V_h$ by the formula
\begin{align*}
(p,q) & = \int_{\Omega} [(I -\frac{1}{\nu}(B(x)) D p(x)] \cdot D q(x)  \ud x.
\end{align*}
Then because
$$
(p,q) = \frac{1}{4} ((p+q,p+q) - (p-q,p-q)),
$$
and using the definition of $\beta$, we get assuming that $|p|_{H^1}=|q|_{H^1}=1$,
\begin{align*}
|(p,q)| & \leq \frac{\beta}{4} (p+q,p+q) + \frac{\beta}{4}  (p-q,p-q) \\
& \leq \frac{\beta}{4} |p+q|_{H^1}^2 + \frac{\beta}{4}  |p-q|_{H^1}^2 = \beta.
\end{align*}
This completes the proof.

\end{proof}

\section{Error analysis of the mixed method for smooth solutions} \label{error}
We will assume without loss of generality that $h \leq 1$.
The goal of this section is to prove the local solvability of \eqref{m11h} for Lagrange elements of degree $k \geq 2$. We define for $\rho >0$, 
$$\bar B_h(\rho)=\{(w_h, \eta_h) \in V_h\times \Sigma_h,\ \| w_h-I_hu\|_{H^1 }\leq \rho,\ \|\eta_h-I_h\sigma\|_{L^{2}}\leq h^{-1}\rho\}.$$
We are interested in elements $(w_h, \eta_h) \in V_h\times \Sigma_h$ satisfying
\begin{equation} \label{discrete-H}
(\eta_h, \tau)+(\div \tau, Dw_h)-\<Dw_h, \tau  n\>=0, \forall  \tau\in \Sigma_h.
\end{equation}
We define
\begin{align*}
\begin{split} 
Z_h & =\{ \, (w_h, \eta_h) \in V_h\times \Sigma_h,  w_h =g_h \,  \text{on} \, \partial \Omega,  (w_h, \eta_h) \ \text{solves} \ \eqref{discrete-H}
 \, \} \, \text{and}
 \end{split}
\end{align*}
\begin{equation*} 
B_h(\rho)=\bar B_h(\rho)\cap Z_h.
\end{equation*}

In \cite{AwanouLiMixed1} the local solvability of \eqref{m11h} was obtained by a fixed point argument which consists in a linearization at the exact solution of \eqref{m1}. To be able to obtain results for quadratic elements we use a time marching method combined with a rescaling argument. 
This is the point of view we took in \cite{Awanou-Std01,Awanou-Std-fd}. We first describe the time marching method at the continuous level.

Let $\nu >0$. We consider the sequence of problems
\begin{align*}
-\nu \Delta u^{r+1} &= -\nu \Delta u^{r} + \det D^2 u^r -f  \, \text{in} \, \Omega \\
u^{r+1} &= g \, \text{on} \, \partial \Omega.
\end{align*}
Put $\sigma^{r+1}=D^2 u^{r+1}$. We obtain the equivalent problems
\begin{align*}
\sigma^{r+1}& =D^2 u^{r+1}  \, \text{in} \, \Omega \\
-\nu \tr \sigma^{r+1} &= -\nu \tr \sigma^{r} + \det \sigma^r -f,  \, \text{in} \, \Omega \\
u^{r+1} &= g \, \text{on} \, \partial \Omega,
\end{align*} 
where $\tr A$ denotes the trace of the matrix $A$.

We are thus lead to consider the sequence of discrete problems:
find $(u_h^{r+1}, \sigma_h^{r+1}) \in V_h \times \Sigma_h$ such that $u_h^{r+1}  = g_h \, \text{on} \, \partial \Omega$ and
\begin{align} 
(\sigma_h^{r+1},\tau) & + (\div \tau, D u_h^{r+1}) - \< D u_h^{r+1}, \tau n \>  = 0, \forall \tau \in \Sigma_h \label{m11h-iter01}\\
-\nu (\tr \sigma^{r+1},v) &= -\nu (\tr \sigma^{m},v) +(\det \sigma_h^{r}-f, v)  , \, \forall v \in V_h \cap H_0^1(\Omega), \label{m11h-iter02}
\end{align}
given an initial guess $(u^0_h,\sigma^0_h)$. We prove below the convergence of $(u_h^{r+1}, \sigma_h^{r+1}) $ to a local solution $(u_h,\sigma_h)$ of the discrete problem \eqref{m11h}. Although \eqref{m11h-iter01}--\eqref{m11h-iter02} may be used in the computations, it is better to use in practice  Newton's method. 

Let $\alpha >0$. We define a mapping $T: V_h\times \Sigma_h\rightarrow V_h\times\Sigma_h$ by 
\begin{eqnarray*}
T(w_h, \eta_h)=(T_1(w_h, \eta_h), T_2(w_h, \eta_h)), 
\end{eqnarray*}
where $T_1(w_h, \eta_h)$ and  $T_2(w_h, \eta_h)$ satisfy
\begin{align}
\begin{split}
(\eta_h-T_2(w_h, \eta_h), \tau)& +(\div\tau, D(w_h-T_1(w_h, \eta_h)))  \\
& \  -\<D(w_h-T_1(w_h, \eta_h)), \tau n\> =(\eta_h, \tau) \\
& \qquad \qquad \qquad  +(\div \tau, Dw_h)-\<Dw_h, \tau n\>, \quad \forall\ \tau\in \Sigma_h \label{eqn.1}
\end{split}
\end{align}
\begin{align}
-\nu (\tr T_2(w_h, \eta_h),v)  & = - \nu (\tr \eta_h,v) + (\det \eta_h - \alpha^2  f, v), \ \forall\ v\in V_h\cap H^1_0(\Omega) \label{eqn.2}\\
T_1(w_h, \eta_h)&=w_h \quad \text{on}\quad \partial\Omega. \label{eqn.3}
\end{align}
Note that \eqref{eqn.1} is equivalent to
\begin{align}
(T_2(w_h, \eta_h), \tau) +(\div\tau, D T_1(w_h, \eta_h) )    -\<D T_1(w_h, \eta_h), \tau n\>  = 0 \  \forall  \ \tau\in \Sigma_h. \label{eqn.11}
\end{align}
Let $I$ denote the $2 \times 2$ identity matrix. We first make the following important observation. 

For $v\in V_h\cap H^1_0(\Omega)$ and $\tau=v I$, we have $\div \tau= D v$ and since $v=0$ on $\partial \Omega$, we have in addition $\tau n=0$ on $\partial \Omega$. Thus using \eqref{eqn.11} 
we obtain
\begin{equation} \label{w-eta-comp}
-\nu (\tr T_2(w_h, \eta_h),v) = - \nu (T_2(w_h, \eta_h), v I) = \nu (D T_1(w_h, \eta_h), D v).
\end{equation}
Similarly, we obtain that if $(w_h,\eta_h)$ solves \eqref{discrete-H}, then 
\begin{equation} \label{w-eta-comp2}
(\tr \eta_h ,v) = -(D w_h, D v), \forall v\in V_h\cap H^1_0(\Omega).
\end{equation}
\begin{lemma} \label{final-lem2}
The mapping $T$ is well defined 
and if $(  \alpha w_h,  \alpha \eta_h)$ is a fixed point  of \eqref{eqn.1}--\eqref{eqn.3} with $w_h=g_h$ on $\partial \Omega$, then $( w_h,  \eta_h)$ solves the nonlinear problem \eqref{m11h}.
\end{lemma}
\begin{proof}
To prove the first assertion, it is enough to prove that if $(w_h,\eta_h) \in V_h \times \Sigma_h$ is such that $w_h  = 0 \, \text{on} \, \partial \Omega$ and
\begin{align*} 
(\eta_h,\tau)  + (\div \tau, D w_h) - \< D w_h, \tau n \>  &= 0, \forall \tau \in \Sigma_h \\
-\nu (\tr \eta_h,v) &=  0 , \, \forall v \in V_h \cap H_0^1(\Omega), 
\end{align*}
then $w_h=0$ and $\eta_h=0$.

Using \eqref{w-eta-comp2}, we obtain $0= - (\tr \eta_h,v)  = (D w_h ,D v)$, for all $v \in V_h\cap H^1_0(\Omega)$. Thus $|w_h|^2_{H^1} =0$. This proves that $w_h=0$ by Poincar\'e's inequality. Using $\tau=\eta_h$ we obtain as well $\eta_h=0$.

The proof of the second assertion is immediate.
\end{proof}

We recall from \cite[Remark 3.6]{AwanouLiMixed1}, see also \cite{Neilan2013,Lakkis11b}, that for $v_h \in V_h$, there exists a unique $\eta_h \in \Sigma_h$ denoted $H(v_h)$, such that 
\begin{equation} \label{disc-H}
(H(v_h), \tau)+(\div \tau, D v_h)-\<D v_h, \tau n\>=0, \forall  \tau\in \Sigma_h,
\end{equation}
 holds. To see this 
 consider the problem: find $\eta_h \in \Sigma_h$ such that
\begin{eqnarray} \label{disc-H2}
(\eta_h, \tau)=-(\div\tau, D v_h )+\<D v_h, \tau n\>, \quad \forall  \tau\in \Sigma_h.
\end{eqnarray}
For $\tau\in \Sigma_h$, we define $F(\tau) = -(\div\tau, D v_h )+\<D v_h, \tau n\>$. Clearly $F$ is linear. By the Schwarz inequality, \eqref{inverse1} and \eqref{trace-inverse}
\begin{align*}
|-(\div\tau, D v_h)+\<D v_h, \tau\cdot n\> |& \leq C ||\tau||_{H^1} || v_h||_{H^1} + C || v_h||_{H^1(\partial \Omega)} ||\tau||_{L^2(\partial \Omega)} \\
& \leq C (h^{-1}  || v_h||_{H^1} +  h^{-\frac{1}{2}} || v_h||_{H^1(\partial \Omega)} ) ||\tau||_{L^2}.
\end{align*}
Thus a unique solution $\eta_h=H(v_h)$ exists by the Lax-Milgram Lemma. 
\begin{rem} \label{H-vh-rem}
From the definition of $H(v_h)$ \eqref{disc-H} and \eqref{disc-H2}, we have for $v_h \in V_h$, 
$$
H(\alpha v_h) = \alpha H(v_h).
$$

\end{rem}
\begin{lemma} \label{est-disc-H}
Let $v_h \in V_h$ such that $||v_h -I_h u||_{H^1} \leq \mu$. Then
$$
||H(v_h) - I_h \sigma||_{L^2} \leq C h^{-1} \mu  +Ch^{k-1}.
$$
\end{lemma}
\begin{proof}

For $\tau \in \Sigma_h$, by \eqref{m11} and \eqref{disc-H} we have
\begin{align*}
(H(v_h) -I_h\sigma, \tau) & = (H(v_h) - \sigma, \tau) + (\sigma -I_h\sigma, \tau)  \\
& = (\sigma-I_h\sigma, \tau)-(\div\tau, D(v_h-u))+\<D( v_h-u), \tau n\> \\
& = (\sigma-I_h\sigma, \tau)-(\div\tau, D(v_h-I_h u))+\<D( v_h- I_h u), \tau n\> \\
& \qquad \qquad -(\div\tau, D(I_h u -u))+\<D( I_h u -u), \tau n\>. 
\end{align*}
Let $\tau=H(v_h) -I_h\sigma$. By the Schwarz inequality, \eqref{inverse1} and \eqref{trace-inverse}
\begin{align*}
\| \tau \|_{L^2}^2 & \leq \|\sigma-I_h\sigma\|_{L^2} \| \tau \|_{L^2}+ C \|\tau\|_{H^1}\|D(v_h - I_hu)\|_{L^2}\\
& \qquad  +C \|D(v_h - I_hu)\|_{L^2(\partial\Omega)}\|\tau \|_{L^2(\partial\Omega)} + C \|\tau\|_{H^1}\|D( I_hu -u )\|_{L^2}\\
& \qquad \qquad \qquad+C \|D(I_h u -u)\|_{L^2(\partial\Omega)}\|\tau \|_{L^2(\partial\Omega)} \\
& \leq \|\sigma-I_h\sigma\|_{L^2} \|\tau \|_{L^2}+ C h^{-1} \mu \|\tau\|_{L^2} +C h^{-1} \|D(v_h - I_hu)\|_{L^2(\Omega)}\|\tau \|_{L^2(\Omega)} \\
& \qquad +C h^{-1} \|\tau\|_{L^2}\| I_hu -u \|_{H^1}\ + C h^{-\frac{1}{2}} \|D(I_hu-u)\|_{L^2(\partial\Omega)}\|\tau\|_{L^2}.
\end{align*}
Therefore
\begin{align*}
\|\tau \|_{L^2} & \leq   Ch^{k+1}+ C h^{-1} \mu  +Ch^{k-1}+Ch^{k-\frac{1}{2}}  \\
& \leq C h^{-1} \mu  +Ch^{k-1}.
\end{align*}
This proves the result.
\end{proof}
It follows from Lemma \ref{est-disc-H}, with $\mu=0$, that $(I_h u, H(I_h u)) \in B_h(\rho)$, i.e. 
 the ball $B_h(\rho)\neq \emptyset$ for 
$\rho = C_0 h^{k}$ for a constant  $C_0 >0$. See also \cite[Lemma 3.5]{AwanouLiMixed1}. As a consequence, see also \cite{Neilan2013},
\begin{equation} \label{disc-H-Ihu}
|| H(I_h u) - I_h \sigma ||_{L^2} \leq C_0 h^{k-1}.
\end{equation}

Let
$$
\tilde{B}_h(\rho) = \{ \, v_h \in V_h, v_h=g_h \, \text{on} \, \partial \Omega, ||v_h-I_hu||_{H^1} \leq \rho \, \},
$$
and consider the mapping
\begin{align*}
\tilde{T}_1: V_h \to V_h, \, \text{defined by} \, \tilde{T}_1(v_h) = T_1(v_h, H(v_h)).
\end{align*}

The motivation to introduce a discrete Hessian $H(v_h)$ in this paper, as opposed to the approach in \cite{AwanouLiMixed1}, is given by Lemma \ref{final-lem0} below. 

\begin{lemma} \label{final-lem0}
If $w_h$ is a fixed point of $\tilde{T}_1$, then $(w_h,H(w_h))$ is a fixed point of $T$ and equivalently, if $(w_h,\eta_h)$ is a fixed point of $T$, then $w_h$ is a fixed point of $\tilde{T}_1$.
\end{lemma}

\begin{proof}
The result was given as \cite[Remark 3.6 ]{AwanouLiMixed1}. Let $w_h$ be a fixed point of $\tilde{T}_1$. We have $T_1(w_h,H(w_h)) = w_h$ and by \eqref{eqn.11} and \eqref{disc-H},  $T_2(w_h,H(w_h)) = H(T_1(w_h,H(w_h))) = H(w_h)$. This proves that $(w_h,H(w_h))$ is a fixed point of $T$.

Conversely if $(w_h,\eta_h)$ is a fixed point of $T$, then $\tilde{T}_1(w_h) = T_1(w_h,H(w_h)) = T_1(w_h,\eta_h) =  w_h$. This completes the proof.
\end{proof}

\begin{lemma} \label{move-ball}
We have for $0 \leq \alpha \leq 1$
\begin{align}
|| \alpha  I_h u - T_1( \alpha I_h u, H( \alpha I_h u) )||_{H^1} & \leq  \frac{C_1}{\nu} \alpha^2  h^{k-1},   \label{u-ball} 
\end{align}
for a positive constant $C_1$. 
\end{lemma}
\begin{proof}
Since $T_1( \alpha I_h u,  H(\alpha I_h u))- \alpha  I_h u=0$ on $\partial \Omega$, by \eqref{w-eta-comp} and \eqref{eqn.2} we have using $w_h= \alpha I_h u$, $\eta_h= H(\alpha I_h u)$ and 
$v=T_1(w_h,\eta_h)-w_h$
\begin{align*}
\nu (D T_1(w_h, \eta_h), D v) & = -\nu (\tr T_2(w_h, \eta_h),v) 
 =  - \nu (\tr \eta_h,v) + (\det \eta_h - \alpha^2 f, v).
\end{align*}
It follows that
\begin{align*} 
\begin{split}
\nu |  D v|_{L^2}^2 & = -\nu (D w_h, D v) - \nu (\tr \eta_h,v) + (\det \eta_h - \alpha^2  f, v).
 \end{split}
\end{align*}
Therefore, using \eqref{w-eta-comp2}, we get
\begin{align} \label{u1}
\begin{split}
\nu |  D v|_{L^2}^2 & =  (\det \eta_h - \alpha^2  f, v).
 \end{split}
\end{align}
On the other hand since $f=\det D^2 u = \det \sigma$, by  \eqref{mean-v} and Remark \ref{H-vh-rem}, on each element $K$
\begin{align} 
\begin{split} \label{partial-est1}
\det \eta_h - \alpha^2 f  & =  \det  H(\alpha I_h u) -  \alpha^2 \det  \sigma =  \det \alpha  H( I_h u) -  \alpha^2 \det  \sigma \\
& = \alpha^2 ( \det H( I_h u) -  \det  \sigma)\\
& =   \alpha^2 (\cof(t H( I_h u) +(1-t) \sigma):(H( I_h u) - \sigma)),
\end{split}
\end{align}
for some $t \in [0,1]$. 

By \eqref{interpol} we have $\|I_h\sigma\|_{L^{\infty}}\leq C\|\sigma\|_{L^{\infty}}$. Thus by \eqref{disc-H-Ihu} and \eqref{inverse0} 
\begin{align*}
|| H(I_h u) ||_{L^{\infty}} & \leq || H(I_h u) - I_h \sigma ||_{L^{\infty}} + \|I_h\sigma\|_{L^{\infty}} \leq C h^{-1} || H(I_h u) - I_h \sigma ||_{L^{2}} + \|I_h\sigma\|_{L^{\infty}} \\
& \leq C h^{k-2} + C\|\sigma\|_{L^{\infty}} \leq C, \, \text{since} \, k \geq 2.
\end{align*}

Thus by \eqref{cof-est} and  \eqref{disc-H-Ihu} 
\begin{align*}
\|\det(H(I_h u))-\det\sigma\|_{L^2(K)} & \leq C\|t H(I_h u)+(1-t)\sigma\|_{L^{\infty}(K)}^{}   \| H(I_h u)-\sigma\|_{L^2(K)} \\
&  \leq C \| H(I_h u) -\sigma\|_{L^2(K)} \\
& \leq C \| H(I_h u) -I_h \sigma\|_{L^2(K)} + C \| I_h \sigma  -\sigma\|_{L^2(K)} \\
& \leq C h^{k-1}.
\end{align*}
Therefore by \eqref{interpol} and \eqref{partial-est1}
\begin{align} \label{u3}
\begin{split}
\| \det \eta_h -  \alpha^2 f  \|_{L^2} = \alpha^2 \| \det(H(I_h u))-\det \sigma \|_{L^2}& \leq C \alpha^2   h^{k-1}.
\end{split}
\end{align}
And so combining \eqref{u1}--\eqref{u3}, \eqref{disc-H-Ihu}, Cauchy-Schwarz inequality, the interpolation error estimate \eqref{interpol} and Poincare's inequality, we get
\begin{align*}
 |v|_{H^1}^2 & \leq \frac{C}{\nu} \alpha^2 h^{k-1} ||v||_{L^2} \leq \frac{C}{\nu} \alpha^2 h^{k-1}  ||v||_{H^1},
\end{align*}
from which \eqref{u-ball} follows.

\end{proof}
We will need  the following lemma
\begin{lemma} \label{zh-lem}
Let  $(w_h, \eta_h) \in Z_h$. Then for a piecewise smooth symmetric matrix field $P$ 
\begin{align}
\begin{split}
((\cof P): \eta_h,v) +  ((\cof P) D w_h , D v )  & 
 \leq  C h ||v||_{H^1}  ||w_h||_{H^1},
\end{split}
\end{align}
for all $v \in V_h \cap H_0^1(\Omega)$ and for a constant $C$ which depends on $||\cof P||_{H^{k+1}(\mathcal{T}_h)}$.
\end{lemma}
\begin{proof}
The proof is the same as the proof of \cite[Lemma 3.7]{AwanouLiMixed1}. There the proof was given for $P=D^2 u$, but it carries over to the general case of this lemma line by line.
The dependence of the constant $C$ on $||\cof P||_{H^{k+1}(\mathcal{T}_h)}$ arises from the use in the proof of the approximation property $||P_{\Sigma_h} (v \cof P) - v \cof P||_{H^m(\mathcal{T}_h)}  \leq C h^{k+1-m} ||v \cof P ||_{H^{k+1}(\mathcal{T}_h)}$. Here $P_{\Sigma_h}$ denotes the $L^2$ projection operator into $\Sigma_h$.

\end{proof}

\begin{lemma} \label{l-inf-est-dis-H}
For $(w_h,\eta_h) \in B_h(\rho), \rho = C_0 h^k$, we have
\begin{align*}
||\eta_h-D^2 w_h||_{L^{\infty}} & \leq C h^{k-2}.
\end{align*}
\end{lemma}

\begin{proof}
Recall that for $(w_h,\eta_h) \in B_h(\rho)$, we have $\eta_h=H(w_h)$. We have by \eqref{inverse0}, \eqref{disc-H-Ihu}
\begin{align*}
||\eta_h-D^2 w_h||_{L^{\infty}} & \leq ||H(w_h)-D^2 w_h||_{L^{\infty}} \\
& \leq  ||H(w_h)-I_h \sigma||_{L^{\infty}} + || I_h \sigma -D^2 w_h||_{L^{\infty}} \\
& \leq C h^{-1}  ||H(w_h)-I_h \sigma||_{L^{2}} + || I_h \sigma -D^2 u||_{L^{\infty}} + || D^2 u -D^2 w_h||_{L^{\infty}} \\
& \leq C h^{k-2} + C h^{k+1} +  || D^2 u -D^2 I_h u||_{L^{\infty}} + || D^2 I_h u -D^2 w_h||_{L^{\infty}} \\
& \leq C h^{k-2} + C h^{-1} ||I_h u - w_h||_{H^1} \\
& \leq C h^{k-2}.
\end{align*}

\end{proof}

The next lemma states a crucial contraction property of the mapping $T_1$ in $\alpha B_h(\rho)$.

\begin{lemma}
Let $(w_1,\eta_1),(w_2,\eta_2) \in B_h(\rho)$ with $\rho \leq \min(C_0,C_{conv}) h^k$. We have
\begin{align} \label{T-contra}
\begin{split}
|T_1(\alpha w_1, \alpha \eta_1)-T_1(\alpha w_2, \alpha \eta_2)|_{H^1}   &  \leq  a  |\alpha w_1- \alpha w_2 |_{H^1},  
\end{split}
\end{align}
for $0 < a < 1$, $h$ sufficiently small, $\alpha=h^{k+2}$ and $\nu=(m+M)/2$.
\end{lemma}

\begin{proof}
Put $v=T_1(\alpha w_1,  \alpha \eta_1)-T_1(\alpha w_2, \alpha \eta_2)$. By assumption $v \in V_h \cap H_0^1(\Omega)$. 
Using \eqref{w-eta-comp} and \eqref{eqn.2} we obtain
\begin{align*}
\nu (D T_1(\alpha w_1, \alpha \eta_1) - D T_1(\alpha w_2, \alpha \eta_2), D v) & = - \nu( \tr T_2(\alpha w_1, \alpha \eta_1) - \tr T_2(\alpha w_2, \alpha \eta_2),v ) \\
& = -\nu (\tr \alpha \eta_1 - \tr \alpha \eta_2,v) + (\det \alpha \eta_1 - \det \alpha \eta_2,v).
\end{align*}
Therefore, using \eqref{mean-v}, we have for some $t \in [0,1]$ and with the notation
$$
Q=t \eta_1 + (1-t) \eta_2 \ \text{and} \ \tir{Q} = t D^2 w_1 + (1-t) D^2 w_2,
$$
\begin{align} \label{contraction-1}
\begin{split}
|v|_{H^1}^2 & = - (\tr \alpha \eta_1 - \tr \alpha \eta_2,v) \\
& \qquad \qquad  + \frac{1}{\nu} ((\cof \alpha (t \eta_1 + (1-t) \eta_2)):\alpha (\eta_1-\eta_2),v) \\
& = \big((-I + \frac{1}{\nu} \cof \alpha Q): \alpha (\eta_1-\eta_2),v\big) \\
& = -(I: \alpha (\eta_1-\eta_2),v) - (D \alpha (w_1-w_2),D v) \\
&  \ +  \frac{1}{\nu} ( (\cof \alpha Q): \alpha (\eta_1-\eta_2),v) + \frac{1}{\nu} ((\cof \alpha Q)D \alpha (w_1-w_2),D v) \\
&    \quad + (D\alpha (w_1-w_2),D v) - \frac{1}{\nu} ((\cof \alpha \tir{Q})D \alpha (w_1-w_2),D v) \\
&  \qquad  + \frac{1}{\nu} ((\cof \alpha \tir{Q})D \alpha (w_1-w_2),D v)  - \frac{1}{\nu} ((\cof \alpha Q)D \alpha (w_1-w_2),D v).
\end{split}
\end{align}
For $(w_1,\eta_1),(w_2,\eta_2) \in B_h(\rho)$, $ t (w_1,\eta_1)+ (1-t) (w_2,\eta_2) \in B_h(\rho)$ and thus for $h$ sufficiently small, by Lemmas \ref{lem-1} and \ref{time-trick} 
we get 
\begin{equation} \label{contraction-12}
| (D(w_1-w_2),D v) - \frac{1}{\nu} ((\cof \alpha \tir{Q} )D(w_1-w_2),D v)| \leq \gamma |w_1-w_2|_{H^1} |v|_{H^1},
\end{equation}
for $0 < \gamma < 1$.

On the other hand, by Lemma \ref{zh-lem}, with $P=I$, we have
\begin{equation} \label{contraction-13}
|-(I:(\eta_1-\eta_2),v) - (D(w_1-w_2),D v)| \leq C h  |w_1-w_2|_{H^1} |v|_{H^1}.
\end{equation}
Applying  Lemma \ref{zh-lem}, with $P=Q$, we get
\begin{align} \label{contraction-14}
\begin{split}
|( (\cof Q):(\eta_1-\eta_2),v) + ((\cof Q)D(w_1-w_2),D v) | & \leq C h ||\cof Q||_{H^{k+1}(\mathcal{T}_h)} \\
& \quad \quad    |w_1-w_2|_{H^1} |v|_{H^1}.
\end{split}
\end{align}
Finally, since by \eqref{cof-mv}
$$
\cof Q - \cof \tir{Q} = \cof(Q-\tir{Q} ) = \cof\bigg(t (\eta_1 - D^2 w_1) + (1-t) (\eta_2 - D^2 w_2) \bigg),
$$
we get using Lemma \ref{l-inf-est-dis-H}
$$
||\cof Q - \cof \tir{Q} ||_{L^{\infty}} \leq C h^{k-2} \leq C, \, \text{since} \, k \geq 2.
$$
Thus
\begin{align} \label{contraction-15}
\begin{split}
\bigg|\frac{1}{\nu} ((\cof \tir{Q})D(w_1-w_2),D v)  - \frac{1}{\nu} ((\cof Q)D(w_1-w_2),D v)\bigg| & \leq C |w_1-w_2|_{H^1} \\
& \quad \qquad  |v|_{H^1}.
\end{split}
\end{align}
We conclude from \eqref{contraction-1}--\eqref{contraction-15} that
\begin{align} \label{contraction-2}
|v|_{H^1} \leq (\gamma + C h + C \alpha h ||\cof Q||_{H^{k+1}(\mathcal{T}_h)} +C \alpha ) |\alpha w_1- \alpha w_2|_{H^1}.
\end{align}
Using the inverse estimate \eqref{inverse2} and noting that $\rho \leq h^2$ 
\begin{align*}
||\cof Q||_{H^{k+1}(\mathcal{T}_h)} & \leq C h^{-k-1} ||\cof Q||_{L^{2}} \leq C h^{-k-1}   ||Q||_{L^{2}}^{} \\
& \leq   C h^{-k-1}   || t \eta_1 + (1-t) \eta_2 ||_{L^{2}}^{} \\
& \leq  C h^{-k-1}  (||\eta_1||_{L^{2}} + ||\eta_2||_{L^{2}} )^{} \\
& \leq C h^{-k-1}  (||\eta_1 - I_h \sigma||_{L^2} + ||\eta_2 - I_h \sigma||_{L^2} + 2 ||I_h \sigma||_{L^2})^{} \\
& \leq C h^{-k-1}  (h^{-1} \rho +  || \sigma||_{L^2})^{} \leq C h^{-k-1} (C h + || \sigma||_{L^2}) \leq C h^{-k-1}. 
\end{align*}
Since $\gamma < 1$, and $\alpha=h^{k+2}$, for $h$ sufficiently small, $C h + C \alpha h ||\cof Q||_{H^{k+1}(\mathcal{T}_h)} +C \alpha < 1-\gamma$. We conclude from 
\eqref{contraction-2} that \eqref{T-contra} holds.
\end{proof}

\begin{lemma} \label{final-lem1}
For $\rho = \min(C_0,C_{conv}) h^k$, the mapping $\tilde{T}_1$ has a unique fixed point in $\alpha \tilde{B}_h(\rho)$ for $\alpha=h^{k+2}$.
\end{lemma}
\begin{proof}
Note that by \eqref{T-contra}, $\tilde{T}_1$ is a strict contraction in $\alpha \tilde{B}_h(\rho)$ for $\rho \leq$ $ \min(C_0,C_{conv}) h^k$. We now show that $\tilde{T}_1$ maps $ \alpha \tilde{B}_h(\rho)$ into itself.
Let $v_h \in \tilde{B}_h(\rho)$. We have by \eqref{T-contra} and \eqref{u-ball} 
\begin{align*}
||\tilde{T}_1(\alpha v_h)- \alpha I_hu||_{H^1} & \leq ||\tilde{T}_1(\alpha v_h)-\tilde{T}_1(\alpha I_h u) ||_{H^1} + ||\tilde{T}_1(\alpha I_h u)- \alpha I_h u ||_{H^1} \\
& \leq a ||\alpha v_h - \alpha I_h u||_{H^1} + C_1 \alpha^2  h^{k-1} \\
& \leq a \alpha \rho + C_1 \alpha h^{2 k+ 1} =  a \alpha \rho + C_1 h^{k+1} \alpha  h^{ k}. 
\end{align*}
Therefore for $h$ sufficiently small, $C_1 h^{k+1} \leq  \min(C_0,C_{conv}) (1-a) $ and so
$$
||\tilde{T}_1(\alpha v_h)- \alpha I_hu||_{H^1} \leq a \alpha \rho + (1-a) \alpha \rho.
$$
The result then follows from the Banach fixed point theorem.
\end{proof}

We can now state the main result of this paper

\begin{thm}
Problem \eqref{m11h} has a unique local solution $(u_h,\sigma_h)$ for $k \geq 2$ and $h$ sufficiently small. We have
\begin{align*}
||u_h - I_h u||_{H^1} & \leq C h^k \\
||\sigma_h - I_h \sigma||_{H^1} & \leq C h^{k-1}.
\end{align*}
\end{thm}

\begin{proof}
Recall that for $(u_h,\sigma_h) \in B_h(\rho)$, we have $\sigma_h=H(u_h)$. 
The result follows from Lemmas \ref{final-lem0}, \ref{final-lem1} and \ref{final-lem2}, the definition of $B_h(\rho)$ and \eqref{disc-H-Ihu}.

The local solution $u_h$ given by Lemma \ref{final-lem1} satisfies $||u_h - I_h u||_{H^1}  \leq C h^k $.  Since by Lemma \ref{final-lem0}, $(u_h,H(u_h))$ is a fixed point of $T$, by Lemma \ref{final-lem2}, $(u_h,H(u_h))$ solves \eqref{m11h}. By the definition of $B_h(\rho)$ $\sigma_h=H(u_h)$ and by \eqref{disc-H-Ihu}, we have $||\sigma_h - I_h \sigma||_{H^1}  \leq C h^{k-1}$.

\end{proof}



\end{document}